\documentclass[12pt]{amsart}

\usepackage[english]{babel}
\usepackage{amsmath}
\usepackage{amsfonts}
\usepackage{amssymb}
\usepackage{amscd}
\usepackage[shortlabels]{enumitem}
\usepackage{geometry}
\usepackage[colorlinks=true,citecolor=blue,linkcolor=blue,urlcolor=blue]{hyperref}
\usepackage[capitalise,noabbrev]{cleveref}
\newtheorem{theorem}{Theorem}[section]
\theoremstyle{plain}
\newtheorem{corollary}[theorem]{Corollary}
\newtheorem{lemma}[theorem]{Lemma}
\newtheorem{proposition}[theorem]{Proposition}
\theoremstyle{definition}

\newtheorem{remark}{Remark}

\numberwithin{equation}{section}
\newcommand{\Sh}{\operatorname{Sh}}
\newcommand{\germ}{\operatorname{germ}_{\infty}}

\begin{document}

\title[Kwasik--Schultz manifolds are $\mathcal{Z}$-compactifiable]{Kwasik--Schultz manifolds are $\mathcal{Z}$-compactifiable}
\author{Shijie Gu}
\address{Department of Mathematics\\
Northeastern University, Shenyang, Liaoning, China, 110004}
\email{shijiegutop@gmail.com}
\date{August 25th, 2026}

\begin{abstract}
Kwasik and Schultz constructed two-ended open $4$-manifolds which
satisfy the usual finiteness and stability conditions at infinity but
do not admit arbitrarily small $1$-neighborhoods.  In particular, neither end is collarable, so the manifolds are not completable.
We show that the open manifolds associated to their non-desuspendable
$C_2$-actions nevertheless admit finite-dimensional compact ANR
$\mathcal{Z}$-compactifications.
Consequently, there exists a $\mathcal{Z}$-compactifiable open
$4$-manifold which is not pseudo-collarable.  This answers a question of
Guilbault and Tinsley.
\end{abstract}

\maketitle

\section{Introduction}

A closed subset $Z$ of an absolute neighborhood retract (ANR) $X$ is a
$\mathcal{Z}$-\emph{set} if, for every open set $U\subseteq X$, the
inclusion $U-Z\hookrightarrow U$ is a homotopy equivalence.  A
compactification $\widehat X=X\sqcup Z$ is called a
$\mathcal{Z}$-\emph{compactification} if $\widehat X$ is a compact ANR and
$Z$ is a $\mathcal{Z}$-set in $\widehat X$.  Here the added set $Z$
need not itself be an ANR.  For compact metric
ANRs, the open-set definition is equivalent to the instantaneous-homotopy
definition \cite[Sec.~2, p.~1266]{AG99}.

An $n$-manifold $M$ is \emph{completable} if there are a compact
$n$-manifold $\widehat M$ and a compact set $C\subset\partial\widehat M$
such that $M\cong\widehat M-C$.  If $M$ is boundaryless, then necessarily
$C=\partial\widehat M$, so completability is equivalent to being the
interior of a compact manifold; in particular its finitely many ends are
collared.  Following the terminology used by Kwasik and Schultz, a
\emph{$1$-neighborhood} of a stable end is a manifold neighborhood $U$ for
which the inclusion $\partial U\hookrightarrow U$ induces an isomorphism
on fundamental groups (and hence identifies this group with the stable end
group in the situation considered here).

Let $V$ be a $4$-dimensional orthogonal representation of $C_2$ whose
action on the unit sphere is free, and write $D(V)$ and $S(V)$ for its
unit disk and unit sphere.  Choose
$\alpha\in L_4^h(C_2;1)$ whose transfer to $L_4^h(1)=8\mathbb Z$ is the
generator $8$.  Kwasik and Schultz construct a compact smooth
$4$-manifold $P^4(V,\alpha)$ from $D(V)$ by attaching equivariant
$2$-handles.  Its boundary
\[
        \Sigma=\Sigma(V,\alpha)=\partial P^4(V,\alpha)
\]
is an integral homology $3$-sphere with a free $C_2$-action, and there is
an equivariant degree-one map $\Sigma\to S(V)$
\cite[proof of Thm.~2.1, p.~448]{KS88}.  Starting from this data, they
construct a non-desuspendable semifree $C_2$-action on $S^4$, which we
denote by $S^4(V,\alpha)$ \cite[Thm.~2.1]{KS88}.  Moreover,
\(
   \mu(\Sigma(V,\alpha))=8 \pmod{16};
\)
see \cite[proof of Thm.~2.1, pp.~448--449]{KS88}.

Put
\(
        Q=\Sigma/C_2,
\)
and let
\(
        \kappa:\pi_1Q\twoheadrightarrow C_2
\)
be the classifying homomorphism of the double cover $\Sigma\to Q$.
Let $p_-$ and $p_+$ denote the two fixed points of the
$C_2$-action on $S^4(V,\alpha)$, and set
\[
    W_\alpha=S^4(V,\alpha)-\{p_-,p_+\}.
\]
Since the action is semifree, $C_2$ acts freely on $W_\alpha$.
We denote the corresponding orbit manifold by
\(
    M_\alpha=W_\alpha/C_2\).
Thus $M_\alpha$ is a two-ended open $4$-manifold. Our main result is
the following.

\begin{theorem}\label{Th: KS Z-compactification}
Each end of $M_\alpha$ has a cofinal closed neighborhood which admits
a finite-dimensional compact ANR $\mathcal{Z}$-compactification.
Consequently, $M_\alpha$ admits a compact ANR
$\mathcal{Z}$-compactification of covering dimension at most four.

For either end, the $\mathcal{Z}$-boundary may be chosen in the form \(Z=Q/C\),
where $Q$ is the quotient of the Kwasik--Schultz homology sphere by its
free involution and $C\subset Q$ is a connected \v{C}ech-acyclic
continuum.  The space $Z$ is a non-ANR integral homology $3$-manifold and
\(\Sh(Z)=\Sh(\mathbb{R}P^3)\).
\end{theorem}

\begin{corollary}\label{Cor: KS non-pseudo-collar}
There exists a $\mathcal{Z}$-compactifiable open $4$-manifold which is not
pseudo-collarable.
\end{corollary}

This answers the question of Guilbault and Tinsley
\cite[p.~285]{GT03} asking whether a $\mathcal{Z}$-compactifiable open
manifold can fail to be pseudo-collarable.

We fix some terminology.  A \emph{$C_2$-marking} on a connected space
$X$ is a chosen homomorphism
\(
        \lambda_X:\pi_1X\to  C_2\).
A map $f:X\to Y$ between $C_2$-marked spaces is \emph{marked} if
$\lambda_Y f_*=\lambda_X$ (with compatible choices of base points), and a
\emph{marked homotopy equivalence} is a homotopy equivalence which is
marked.  When $\lambda_X$ is onto, the associated connected double cover
will be called the \emph{reference double cover}.  We put
$R=\mathbb Z[C_2]$ and use the notation
\[
        H_*(X;R):=H_*(\widetilde X;\mathbb Z)
\]
for the ordinary integral homology of the reference double cover
$\widetilde X\to X$; this is the usual homology with coefficients in the
regular $\mathbb Z[C_2]$-module.

For an end of a space $X$, its \emph{homeomorphism germ at infinity} is
represented by a cofinal closed neighborhood of that end, with two such
neighborhoods representing the same germ when sufficiently small cofinal
subneighborhoods are homeomorphic.  We write $\germ X$ when the end under
consideration is clear.  In the terminology of Freedman and Quinn, an end
of a manifold $M$ is \emph{tame} if some closed neighborhood $V$ of the end
admits a proper map
\[
        h:V\times(0,1]\longrightarrow M,
        \qquad h(v,1)=v.
\]
A \emph{weak collar} is such a neighborhood for which the proper map may
be chosen to take values in $V$ itself,
\[
        h:V\times(0,1]\longrightarrow V,
        \qquad h(v,1)=v;
\]
see \cite[Sections~11.9A--11.9B]{FQ90}.  If the fundamental group of a weak collar
is identified with $C_2$, its \emph{weak-collar data} are
$(\partial V,\pi_1\partial V\to C_2)$ as in
\cite[Section~11.9C]{FQ90}.  An \emph{$R$-homology $h$-cobordism of data}
between $(N_0,\lambda_0)$ and $(N_1,\lambda_1)$ is a compact cobordism
$L$ carrying a $C_2$-marking which restricts to the given boundary
markings and satisfies
\[
        H_*(L,N_0;R)=H_*(L,N_1;R)=0.
\]

We write $\Sh(X)$ for the (Borsuk) shape of a compactum $X$.  A
compactum $X$ is \emph{cell-like} if $\Sh(X)=\Sh(\{*\})$; a map is
cell-like if each point inverse is cell-like.  A compactum $C$ is
\emph{nearly $1$-movable} if, for each neighborhood $U$ of $C$, there is a
neighborhood $V\subset U$ such that, for every neighborhood $T\subset V$,
each loop in $V$ can be homotoped in $U$ to a loop in $T$; base points may
move.  This is the formulation used by Shrikhande
\cite[pp.~121--123]{Shr83}.  Finally, by an \emph{integral homology
$3$-manifold} we mean a locally compact space having the local singular
homology groups of $\mathbb R^3$; no ANR hypothesis is included.

We first describe the proof.  Delete a small invariant ball from
$P^4(V,\alpha)$ and take the free quotient.  The cocores of the quotient
$2$-handles give pairwise disjoint framed proper disks whose boundary
circles normally generate $\ker\kappa$.  In Section~3 these disks are used
to construct a nested intersection
\(
        F=\cap_i B_i
\)
of relative $4$-balls.  The compactum $F$ is cell-like, and deleting it
preserves both the $C_2$-marked fundamental group and homology with
$R=\mathbb Z[C_2]$ coefficients.  Collapsing $F$ therefore produces a
compact ANR with a $\mathcal{Z}$-boundary.

The compactum $F$ is constructed in the punctured $2$-handlebody coming
from $P^4(V,\alpha)$, not in the periodic cobordism used by Kwasik and
Schultz to build the actual end.  To pass from this auxiliary construction
to the Kwasik--Schultz end, we use the weak-end theorem and the
classification of weak collars in Freedman and Quinn
\cite[Sections~11.9B--11.9C]{FQ90}.  The complementary cobordism is an
$R$-homology $h$-cobordism of weak-collar data, so the two ends have
homeomorphic germs.  The signature-bearing homology remains in the compact
part of the punctured $2$-handlebody.

Section~2 distinguishes the punctured Kwasik--Schultz $2$-handlebody from
the periodic block and records the cocore disks used later.  In Section~3
we construct the nested relative cap, form the compact ANR quotient, and
determine the local and shape properties of its $\mathcal{Z}$-boundary.
Section~4 compares the resulting auxiliary end with the actual
Kwasik--Schultz end and completes the proof.

\section{The compact Kwasik--Schultz piece}

We begin with the compact part of the Kwasik--Schultz construction that
will be used below.  Recall that
\[
 \Sigma=\Sigma(V,\alpha),\qquad Q=\Sigma/C_2,
 \qquad
 K=\ker\bigl(\kappa:\pi_1Q\twoheadrightarrow C_2\bigr)=\pi_1\Sigma.
\]
The manifold $\Sigma$ is an integral homology $3$-sphere.  It follows that
$K$ is perfect, and it is finitely generated because $\Sigma$ is compact.
Moreover, $K\ne 1$: otherwise $\Sigma$ would be a homotopy $3$-sphere and
would have trivial Rochlin invariant, contrary to the Kwasik--Schultz
calculation \cite[proof of Thm.~2.1, pp.~448--449]{KS88}.

There are two different compact $4$-dimensional pieces in the
Kwasik--Schultz construction, and we will use both.  First, let $B\subset
P^4(V,\alpha)$ be a small invariant linear ball about its unique fixed
point and put
\[
 A=\bigl(P^4(V,\alpha)-\operatorname{Int}B\bigr)/C_2.
\]
Then $A$ is a compact cobordism with
\[
 \partial A=P_0\sqcup Q,
 \qquad
 P_0=S(V)/C_2\cong\mathbb{R}P^3.
\]
These two boundary components should not be confused: $P_0$ is the
quotient of the small linear sphere around the fixed point, whereas
$Q=\Sigma/C_2$ is the quotient of the Kwasik--Schultz homology sphere.
In particular, they are not homeomorphic in the present example.  The
reference double cover of $A$ is
\[
                \widetilde A=P^4(V,\alpha)-\operatorname{Int}B.
\]
The involution is orientation preserving.  Indeed, its linear action on
$S(V)=S^3$ is antipodal, and its orientation character is already
determined near the fixed point.  Thus $A$ and $Q$ are oriented, and the
reference cover $\widetilde A$ is oriented.

The second piece is the periodic block $T(V,\alpha)$.  Kwasik and Schultz
obtain it by performing equivariant surgery, away from the boundary, on
\[
        \Sigma\times I\longrightarrow S(V)\times I.
\]
Consequently, both boundary components of $T(V,\alpha)$ are copies of
$\Sigma$, and both boundary components of $T(V,\alpha)/C_2$ are copies of
$Q$.  It is copies of this latter cobordism, not copies of $A$, that are
concatenated to form the two-ended Kwasik--Schultz manifold; see
\cite[p.~448 and Cor.~3.4]{KS88}.  The role of $A$ below is different: its
$2$-handle cocores provide the disk system from which we construct the
cell-like compactum.

\begin{remark}\label{Rem: KS concatenation}
The definition of $M_\alpha$ in terms of the two fixed points could
obscure its geometric structure.  In the Kwasik--Schultz construction,
$T(V,\alpha)$ is a compact cobordism whose two boundary components are
copies of the homology sphere $\Sigma(V,\alpha)$.  By concatenating
copies of this cobordism in both directions, one obtains
\[
 T_\infty(V,\alpha)
 =
 \cdots\cup T_{-1}\cup T_0\cup T_1\cup\cdots ,
 \qquad T_i\cong T(V,\alpha).
\]
The endpoint compactification of $T_\infty(V,\alpha)$ is
$S^4(V,\alpha)$, with the two added endpoints corresponding to the
fixed points $p_-$ and $p_+$.  Thus
\[
 S^4(V,\alpha)-\{p_-,p_+\}
 \cong T_\infty(V,\alpha),
\]
and hence
\[
 M_\alpha\cong T_\infty(V,\alpha)/C_2.
\]
If
\[
 Q=\Sigma(V,\alpha)/C_2,
 \qquad
 \overline T=T(V,\alpha)/C_2,
\]
then $\overline T$ is a compact cobordism from $Q$ to $Q$, and
$M_\alpha$ may equivalently be viewed as the bi-infinite
concatenation
\[
 M_\alpha\cong
 \cdots\cup_Q\overline T\cup_Q\overline T
 \cup_Q\overline T\cup_Q\cdots .
\]
See \cite[proof of Theorem~2.1, p.~448; Corollary~3.4,
pp.~450--451]{KS88}.
\end{remark}

\begin{proposition}\label{Prop: compact piece cocores}
The manifold $A$ has fundamental group $C_2$.  Furthermore, there are
pairwise disjoint neat locally flat framed proper disks
\begin{equation}\label{Eq: compact piece disks}
                    D_1,\ldots,D_s\subset A,
                    \qquad \partial D_j\subset Q,
\end{equation}
whose boundary classes normally generate $K$ in $\pi_1Q$.
\end{proposition}

\begin{proof}
Kwasik and Schultz construct $P^4(V,\alpha)$ from $D(V)$ by attaching a
finite equivariant family of $2$-handles; see the proof of
\cite[Thm.~2.1, p.~448]{KS88}.  Adding $2$-handles to a $4$-ball does not
create fundamental group, and deleting an interior $4$-ball does not
change fundamental group.  Thus $\widetilde A$ is simply connected.  The
action on $\widetilde A$ is free, so
\(\pi_1A\cong C_2\).

The $C_2$-action permutes the attached $2$-handles in free orbits.
For each $C_2$-orbit of $2$-handles, the union of the two cocores descends
to a neat locally flat framed proper disk in $A$ whose boundary lies in
$Q$.  Distinct handle orbits give disjoint disks.  Turn
the quotient handle decomposition upside down.  Reading it from $Q$
toward $P_0$, the attaching circles of the relative $2$-handles are
precisely the boundary circles in \eqref{Eq: compact piece disks}.  The Seifert--van
Kampen theorem gives
\begin{equation}\label{Eq: cocore normal quotient}
 \pi_1Q/\left\langle\!\left\langle
      \partial D_1,\ldots,\partial D_s
 \right\rangle\!\right\rangle
       \cong\pi_1A\cong C_2.
\end{equation}
The homomorphism in \eqref{Eq: cocore normal quotient} is the classifying
homomorphism $\kappa$.  Therefore
\[
 \left\langle\!\left\langle
      [\partial D_1],\ldots,[\partial D_s]
 \right\rangle\!\right\rangle_{\pi_1Q}=K,
\]
as required.
\end{proof}

\section{A relative cap}

The term \emph{relative cap} will be used locally for the nested compactum
constructed in this section.  By a \emph{neat relative $4$-ball} in a
pair $(U,Q)$, where $Q$ is a boundary component of $U$, we mean a neatly
embedded copy $B\cong D^4$ such that $H=B\cap Q$ is a codimension-zero
submanifold of $\partial B$.  In our construction $H$ is a handlebody and
the complementary piece
\[
        J=\overline{\partial B-\operatorname{Int}H}
\]
is the complementary handlebody in a Heegaard splitting of $\partial B$.
The \emph{relative cap} will be the intersection $F=\bigcap_i B_i$ of a
nested sequence of such relative $4$-balls.  This terminology is only a
convenient name for the construction below.

We use the commutator convention $[a,b]=a^{-1}b^{-1}ab$.

\begin{lemma}\label{Lemma: finite commutator system}
Let $K\ne 1$ be a finitely generated perfect group.  There are a finite
ordered list $\ell_1,\ldots,\ell_m\in K$, containing a generating set for
$K$, and words
\begin{equation}\label{Eq: basis commutator words}
 w_j=\prod_{t=1}^{r_j}[x_{p(j,t)},x_{q(j,t)}]
       \in F(x_1,\ldots,x_m)',
 \qquad r_j\geq1,\quad p(j,t)\ne q(j,t),
\end{equation}
such that
\begin{equation}\label{Eq: group commutator equations}
 \ell_j=\prod_{t=1}^{r_j}
       [\ell_{p(j,t)},\ell_{q(j,t)}]
\end{equation}
for every $j$.  Consequently, the homomorphisms
\begin{equation}\label{Eq: rho phi}
 \rho:F_m\twoheadrightarrow K,\quad \rho(x_j)=\ell_j,
 \qquad
 \phi:F_m\longrightarrow F_m,\quad \phi(x_j)=w_j,
\end{equation}
satisfy
\begin{equation}\label{Eq: rho phi identities}
 \rho\phi=\rho,\qquad \phi(F_m)\subseteq F_m',
 \qquad \rho\phi^r=\rho\quad (r\geq 0).
\end{equation}
\end{lemma}

\begin{proof}
Choose generators $s_1,\ldots,s_d$ for $K$.  Since $K=[K,K]$, each $s_i$
is a finite product of commutators $[a_{it},b_{it}]$.  Let $L$ be a finite
list containing the nonidentity elements $s_i$ and all of the nonidentity
elements $a_{it}$ and $b_{it}$ which occur in these expressions; trivial
commutators are omitted.  The subgroup generated by the
pairwise commutators of entries of $L$ contains every $s_i$, and hence is
all of $K$.  It follows that every entry of $L$ is itself a product of
pairwise commutators of entries of $L$.  Replace $[a,b]^{-1}$ with $[b,a]$
when necessary, and repeat entries of $L$ so that distinct indices may be
used in every nontrivial commutator.  Taking the resulting list as
$\ell_1,\ldots,\ell_m$ proves \eqref{Eq: basis commutator words} and
\eqref{Eq: group commutator equations}.  Equations
\eqref{Eq: rho phi}--\eqref{Eq: rho phi identities} follow immediately.
\end{proof}

We use the following elementary $4$-dimensional construction to iterate
the algebra of Lemma \ref{Lemma: finite commutator system}.

\begin{lemma}[Rectangular word-box lemma]\label{Lemma: rectangular word box}
Let $\mathcal A=A_1\sqcup\cdots\sqcup A_m\subset S^3$ be the standard
unlink and identify
\[
        \pi_1(S^3-\mathcal A)=F(x_1,\ldots,x_m)
\]
by means of based meridians.  Given the words $w_j$ in
\eqref{Eq: basis commutator words}, there is a split zero-framed unlink
$\mathcal C=C_1\sqcup\cdots\sqcup C_m\subset S^3-\mathcal A$ such that
\begin{equation}\label{Eq: literal word realization}
                         [C_j]=w_j
             \quad\text{in }\pi_1(S^3-\mathcal A).
\end{equation}
If $\Delta_j\subset D^4$ are the standard pushed-in disks bounded by the
$C_j$, then
\begin{equation}\label{Eq: old curves die}
 [A_k]=1\quad\text{in}\quad
 \pi_1\left(D^4-\bigcup_j\Delta_j\right)
 \qquad (1\leq k\leq m).
\end{equation}

Moreover, the two unlinks may be equipped with connecting trees so that,
if $B_1$ is a regular neighborhood of the disks $\Delta_j$ and the second
tree in $B_0=D^4$, then $B_1$ is a $4$-ball and there is a
homeomorphism
\[
 (B_0;H_0,J_0)\longrightarrow(B_1;H_1,J_1)
\]
which preserves the chosen ordered meridian and core coordinates on the
boundary.  Here $(H_i,J_i)$ is the standard genus-$m$ Heegaard splitting
of $\partial B_i$.  If
\[
                   Y_0=\overline{B_0-\operatorname{Int}B_1},
\]
then, with the induced free-group coordinates,
\[
 \pi_1J_1\xrightarrow{\cong}\pi_1Y_0,
 \qquad
 \pi_1J_0\longrightarrow\pi_1Y_0\text{ is trivial},
\]
and the inclusion $H_1\hookrightarrow H_0$ induces the homomorphism
$\phi$ in \eqref{Eq: rho phi}.
\end{lemma}

\begin{proof}
Choose a product ball $D^2\times[0,1]\subset S^3$ in which every component
$A_k$ appears as a vertical strand $\{a_k\}\times[0,1]$.  Close the
strands by the standard trivial tangle in the complementary ball.  Choose
disjoint product boxes
\[
 D^2\times I_1,\ldots,D^2\times I_m,
 \qquad I_1<I_2<\cdots<I_m,
\]
so that every strand passes once through every box and in the same order.

Inside the $j$th box, place one Borromean axis around the two strands
indexed by $p(j,t)$ and $q(j,t)$ for every commutator occurring in $w_j$.
With product base paths, the corresponding axis represents the chosen
word $[x_{p(j,t)},x_{q(j,t)}]$; this is the standard Borromean
calculation, cf. \cite[Sec.~4.2 and Fig.~4.4]{CFT09}.  Join these axes in
their prescribed order by disjoint untwisted product bands.  The resulting
circle is $C_j$.  It bounds an embedded disk in the $j$th box, since it is
a boundary connected sum of unknotted axes.  Thus all the $C_j$ form a
split zero-framed unlink.  The two sides of each band contribute inverse
base paths, and hence the band-sum calculation gives
\eqref{Eq: literal word realization} with the indicated base paths.

We verify \eqref{Eq: old curves die}.  Fix $A_k$ and one word box.  Join
the endpoints of the arc $A_k$ in this box by a reference arc on the
boundary of the box.  Since $w_j$ is a product of commutators, the closed
curve obtained from these two arcs has linking number zero with $C_j$.
The curve $C_j$ is an unknot in the box, so linking number identifies the
fundamental group of its complement with $\mathbb Z$.  Therefore the
original arc is homotopic, rel endpoints, to the reference boundary arc in
the complement of $C_j$.  Perform this operation in all the disjoint word
boxes.  The resulting representative of $A_k$ lies in the complement of
the interiors of finitely many disjoint $3$-balls, which is simply
connected.  Thus $[A_k]=1$ in $\pi_1(S^3-\mathcal C)$.  Boundary inclusion
induces an isomorphism from this group to the fundamental group of the
standard disk exterior in $D^4$, proving \eqref{Eq: old curves die}.

Choose a planar connecting tree $T_{\mathcal A}$ for $\mathcal A$, with
interior disjoint from $\mathcal C$.  To choose the second tree while
retaining the ordered boundary coordinates, begin with the standard planar handcuff
graph for an $m$-component unlink and carry it to $\mathcal C$ by an
ambient homeomorphism.  Perturb the interior of its tree, keeping the
ordered link components fixed, off the one-complex
$\mathcal A\cup T_{\mathcal A}$.  Denote the resulting tree by
$T_{\mathcal C}$.  Put
\[
 J_0=N(\mathcal A\cup T_{\mathcal A}),
 \qquad
 H_0=\overline{\partial B_0-\operatorname{Int}J_0}.
\]
Choose the first neighborhood sufficiently thin that
\[
 \mathcal C\cup T_{\mathcal C}\subset\operatorname{Int}
 \overline{S^3-N(\mathcal A\cup T_{\mathcal A})}
   =\operatorname{Int}H_0.
\]
A regular
neighborhood of either handcuff graph is a standard genus-$m$ handlebody.
Push the interiors of the disks bounded by $\mathcal C$ into $B_0$,
retaining the tree on the boundary, and take a thin regular neighborhood.
A neighborhood of each disk is a $4$-ball.  The thickened edges of the
tree join different components, so their boundary connected sum $B_1$ is
again a $4$-ball.  Under the natural identification,
\[
 H_1=B_1\cap\partial B_0=N(\mathcal C\cup T_{\mathcal C}),
 \qquad
 J_1=\overline{\partial B_1-\operatorname{Int}H_1}.
\]

The disk exterior $Y_0$ has free fundamental group generated by the normal
meridians of the disks $\Delta_j$.  Those same meridians are the ordered
core loops of $J_1$, so $\pi_1J_1\to\pi_1Y_0$ is an isomorphism.  The
ordered core loops of $J_0$ are the curves $A_k$, and
\eqref{Eq: old curves die} shows that their images in $\pi_1Y_0$ are
trivial.  Deleting a thin neighborhood of the connecting tree does not
change these calculations, since loops and null-homotopy disks in a
$4$-manifold may be put in general position off a $1$-complex.

Finally, each $(H_i,J_i)$ is a standard genus-$m$ Heegaard splitting
equipped with an ordered meridian/core system.
Choose an orientation-preserving boundary homeomorphism which carries the
ordered meridian and core system of the first splitting to that of the
second.  It extends across the two $4$-balls by the Alexander trick.  Under
the resulting identifications $\pi_1H_0\cong F_m\cong\pi_1H_1$, the curves
$C_j$ give the generators on the inner side, so
\eqref{Eq: literal word realization} says that the induced homomorphism
$\pi_1H_1\to\pi_1H_0$ is precisely $\phi$.  This completes
the proof.
\end{proof}

\begin{proposition}\label{Prop: relative cap}
Let $U$ be a compact oriented topological $4$-manifold and let $Q$ be a
component of its boundary.  Suppose that
\[
 \pi_1Q\longrightarrow\pi_1U
       \xrightarrow[\cong]{\lambda} C_2
\]
is onto with nontrivial finitely generated perfect kernel $K$.  Suppose,
in addition, that $U$ contains pairwise disjoint neat locally flat framed
proper disks whose boundaries lie in $Q$ and normally generate $K$ in
$\pi_1Q$.

Then there are nested neat relative $4$-balls
\[
 B_{i+1}\subset\operatorname{Int}_U B_i\quad(i\geq0),
 \qquad B_i\cong D^4,
 \qquad H_i=B_i\cap Q,
\]
such that, for
\[
                     F=\bigcap_iB_i,
                     \qquad C=F\cap Q=\bigcap_iH_i,
\]
the following conclusions hold:
\begin{enumerate}[(a)]
\item $F$ is cell-like, while $C$ is connected and \v{C}ech-acyclic
with arbitrary constant coefficients;
\item
\begin{equation}\label{Eq: external image K}
        \operatorname{im}(\pi_1H_i\longrightarrow\pi_1Q)=K
        \quad\text{for every }i;
\end{equation}
\item
\[
        \pi_1(\operatorname{Int}_UB_i-F)=1
        \quad\text{for every }i;
\]
\item the inclusion
\begin{equation}\label{Eq: marked deletion equivalence}
                         U-F\longrightarrow U
\end{equation}
is a marked homotopy equivalence; and
\item $C$ is not nearly $1$-movable.
\end{enumerate}
Moreover, if $O_i=\operatorname{Int}_UB_i$, then
\begin{equation}\label{Eq: local deletion equivalence}
                         O_i-F\longrightarrow O_i
\end{equation}
is a homotopy equivalence for every $i$.
\end{proposition}

\begin{proof}
Apply Lemma \ref{Lemma: finite commutator system} to $K$.  Since the
boundaries of the given disks normally generate $K$, normal parallel
copies and boundary bands produce pairwise disjoint framed proper disks
whose boundary classes, with the chosen base paths, are the elements
$\ell_1,\ldots,\ell_m$.  Indeed, represent the finitely many conjugating
paths by narrow pairwise disjoint arc cores in $Q$, attach boundary bands,
and push the interiors of the bands into a collar of $Q$.  Local twists of
the bands correct the framings.  Join regular neighborhoods of these disks
by a boundary tree and denote the resulting relative regular neighborhood
by $B_0$.  It is a $4$-ball.  The boundary piece
\[
        H_0=B_0\cap Q
\]
and the complementary frontier
\[
        J_0=\overline{\partial B_0-\operatorname{Int}H_0}
\]
form the standard genus-$m$ Heegaard splitting of $\partial B_0$.  The
ordered disks and chosen base paths identify $\pi_1H_0$ with $F_m$, and
under this identification the inclusion $H_0\hookrightarrow Q$ induces
\[
        F_m\xrightarrow{\rho}K\hookrightarrow\pi_1Q,
\]
where $\rho$ is the epimorphism defined in \eqref{Eq: rho phi}.

Transfer the construction of Lemma \ref{Lemma: rectangular word box} to
the Heegaard pair $(\partial B_0;H_0,J_0)$ with these ordered coordinates.
Let
\[
                 e:(B_0;H_0,J_0)\longrightarrow(B_1;H_1,J_1)
\]
be the resulting homeomorphism onto the inner relative ball.  Taking the
regular neighborhoods sufficiently thin, we may and do arrange
$B_1\subset\operatorname{Int}_U B_0$.  Define
\[
 B_i=e^i(B_0),\qquad H_i=e^i(H_0),\qquad J_i=e^i(J_0),
 \qquad
 Y_i=e^i\bigl(\overline{B_0-\operatorname{Int}B_1}\bigr).
\]
Transport the ordered free-group coordinates by $e$.  Under the resulting
identifications $\pi_1H_i\cong F_m$, the inclusion
$H_{i+1}\hookrightarrow H_i$ induces the endomorphism
$\phi:F_m\to F_m$ from \eqref{Eq: rho phi}.  Hence
\[
\begin{aligned}
 \operatorname{im}(\pi_1H_i\longrightarrow\pi_1Q)
   &=\operatorname{im}\bigl(
       F_m\xrightarrow{\phi^i}F_m\xrightarrow{\rho}K
       \hookrightarrow\pi_1Q\bigr)\\
   &=\rho\phi^i(F_m)=K,
\end{aligned}
\]
where the last equality follows from $\rho\phi^i=\rho$.  This proves
\eqref{Eq: external image K}.

We next calculate the fundamental group of $\operatorname{Int}_U B_i-F$.
For $j>i$, put
\[
 Z_{i,j}=\overline{B_i-\operatorname{Int}B_j}.
\]
For one shell, Lemma \ref{Lemma: rectangular word box} gives
\begin{equation}\label{Eq: shell alpha beta repeated}
 \pi_1J_{i+1}\xrightarrow{\cong}\pi_1Y_i,
 \qquad
 \pi_1J_i\longrightarrow\pi_1Y_i\text{ trivial}.
\end{equation}
Suppose that $\pi_1Z_{i,j}$ is identified with the free group coming from
the last shell.  Adding $Y_j$ along $J_j$ gives the pushout
\[
\begin{CD}
 \pi_1J_j @>{\cong}>> \pi_1Z_{i,j}\\
 @VV{0}V                  @VVV\\
 \pi_1Y_j @>>> \pi_1Z_{i,j+1}.
\end{CD}
\]
It follows that $\pi_1Z_{i,j+1}\cong\pi_1Y_j\cong F_m$, whereas the
homomorphism $\pi_1Z_{i,j}\to\pi_1Z_{i,j+1}$ is trivial.  Passing to the
direct limit gives the desired group.  More precisely, connected open
collar thickenings of the finite shells $Z_{i,j}$ are cofinal in
$\operatorname{Int}_UB_i-F$ and have the same fundamental groups and
bonding homomorphisms.  Hence
\begin{equation}\label{Eq: carrier direct limit}
 \pi_1(\operatorname{Int}_UB_i-F)
       \cong\underset{j>i}{\operatorname{colim}}\,\pi_1Z_{i,j}=1.
\end{equation}

The sets $H_i$ form a cofinal neighborhood system of $C$ in $Q$.  Since a
handlebody has the homotopy type of a finite bouquet of circles and
$\phi$ is zero on abelianization, continuity of \v{C}ech cohomology gives
\[
                  \widetilde{\check H}^{\,*}(C;G)=0
\]
for every abelian coefficient group $G$.  The continuum $C$ is connected
because it is the intersection of a nested sequence of connected compacta.
Similarly, strict nesting and compactness show that the relative balls
$B_i$ form a neighborhood basis of $F$ by contractible sets.  Hence $F$
has the shape of a point and is cell-like.  We do not need, and do not
assert, that $F$ is cellular in the ambient manifold.
This proves (a).

It remains to prove the global deletion equivalence
$U-F\longrightarrow U$.  Put
\[
        X_i=U-\operatorname{Int}_UB_i,
        \qquad G_i=\pi_1X_i.
\]
The intersection of $X_i$ and $B_i$ is a collar of $J_i$.  Since $B_i$ is
simply connected, the Seifert--van Kampen theorem gives
\begin{equation}\label{Eq: ambient kernel stage}
 G_i\twoheadrightarrow\pi_1U=C_2,
 \qquad
 \ker(G_i\to C_2)
   =\left\langle\!\left\langle
      \operatorname{im}\pi_1J_i
     \right\rangle\!\right\rangle.
\end{equation}
The shell $Y_i$ lies in $X_{i+1}$, and
\eqref{Eq: shell alpha beta repeated} kills every element of $\pi_1J_i$
there.  Hence the entire kernel in \eqref{Eq: ambient kernel stage} maps
trivially to $G_{i+1}$.  All maps commute with the fixed quotient to $C_2$,
and consequently
\begin{equation}\label{Eq: ambient pi1 limit}
             \pi_1(U-F)=\underset{i}{\operatorname{colim}}\,G_i
                \xrightarrow{\cong}C_2=\pi_1U.
\end{equation}
One may use connected open collar thickenings of the $X_i$ in this direct
limit; they have the same fundamental groups and their union is $U-F$.
Equation \eqref{Eq: ambient kernel stage} also shows directly that every
element in the kernel dies at the next stage.

Use the $C_2$-marking to take $R=\mathbb Z[C_2]$ coefficients, or
equivalently pass to the reference double cover.  The reference cover
splits over $F$.  Concretely, every simply connected ball $B_i$
lifts to two copies; equivalently,
$\check H^1(F;\mathbb Z/2)=0$.  Thus the lifted pair
$(\widetilde F,\widetilde C)$ is a disjoint union of two copies of $(F,C)$.
Strong excision and relative Alexander--Lefschetz duality
\cite[Cor.~16.19]{Bre97} give
\begin{equation}\label{Eq: relative duality deletion}
 H_k(U,U-F;R)
  \cong
 \check H^{4-k}(\widetilde F,\widetilde C;\mathbb Z)=0.
\end{equation}
The last equality follows from the long exact sequence of the pair, since
$F$ is cell-like and $C$ is connected and \v{C}ech-acyclic.  Equation \eqref{Eq: relative duality deletion} gives only the homology
statement; the independent fundamental-group calculation
\eqref{Eq: ambient pi1 limit} is needed to control $\pi_1$.

The map induced by \eqref{Eq: marked deletion equivalence} on universal
covers is now an integral homology equivalence.  Both spaces are ANRs and
have CW homotopy type, so the universal-cover version of Whitehead's
theorem proves \eqref{Eq: marked deletion equivalence}.  The same argument
inside $O_i$ gives
\[
 H_k(O_i,O_i-F;\mathbb Z)
   \cong\check H^{4-k}(F,C;\mathbb Z)=0.
\]
The space $O_i$ is contractible and
$\pi_1(O_i-F)=1$ by \eqref{Eq: carrier direct limit}.  Whitehead's theorem
therefore proves \eqref{Eq: local deletion equivalence}.

Finally, put
\begin{equation}\label{Eq: normal closure tower}
 N_r=\left\langle\!\left\langle
        \phi^r(F_m)
       \right\rangle\!\right\rangle_{F_m}.
\end{equation}
Then $N_{r+1}\subseteq N_r$, while
\begin{equation}\label{Eq: normal tower derived}
          N_r\subseteq F_m^{(r)},\qquad \rho(N_r)=K\ne1.
\end{equation}
Since a free group is residually solvable, this tower is not eventually
constant.  More precisely, if $C$ were nearly $1$-movable in the fixed
outer neighborhood $H_0$, some $H_j$ would work for every sufficiently
deep $H_k$.  The corresponding loop homotopies in $H_0$ would imply
\[
                  \phi^j(F_m)\subseteq N_k,
\]
and hence $N_j\subseteq N_k$.  The reverse inclusion follows from the
nesting, so $N_j=N_k$ for all sufficiently large $k$, a contradiction.
Thus $C$ is not nearly $1$-movable, completing the proof.
\end{proof}

Apply Proposition \ref{Prop: relative cap} to the disk system in
Proposition \ref{Prop: compact piece cocores}, with $U=A$.  From now on,
$B_i,H_i,F$ and $C$ will denote the resulting objects.

\subsection{The compact ANR quotient}

Collapse only the relative cap $F$ and write
\[
 q:A\longrightarrow Y=A/F,
 \qquad p=q(F),
 \qquad Z=q(Q)\cong Q/C.
\]

\begin{proposition}\label{Prop: compact quotient}
The space $Y$ is a compact ANR of covering dimension at most four, and
$Z$ is a $\mathcal{Z}$-set in $Y$.  Consequently,
\begin{equation}\label{Eq: W definition}
                W=Y-Z\cong A-(F\cup Q)
\end{equation}
has the $\mathcal{Z}$-compactification $Y=W\sqcup Z$.  Moreover, the
natural inclusion
\begin{equation}\label{Eq: W to A}
                         W\longrightarrow A
\end{equation}
is a marked homotopy equivalence.
\end{proposition}

\begin{proof}
The decomposition defining $q$ has only one nondegenerate element, so its
equivalence relation is closed.  Therefore $Y$ is compact and metrizable.
Let $K_1\subset K_2\subset\cdots$ be a compact exhaustion of $A-F$.  The
sets $q(K_i)$ are closed in $Y$, each has covering dimension at most four,
and
\[
                         Y=\{p\}\cup\bigcup_iq(K_i).
\]
The countable closed-sum theorem for covering dimension
\cite{HW41} gives $\dim Y\leq4$.  The quotient map $q$ is cell-like.
Lacher's finite-dimensional image theorem
\cite[Cor.~3.3, p.~725]{Lac69} now implies that $Y$ is a compact ENR, and hence a
compact ANR.  Lacher's proper cell-like mapping theorem
\cite[Thm.~1.2 and Cor.~1.3, pp.~719--720]{Lac69} also shows that $q$ is a homotopy
equivalence.

We next prove that $p$ is a $\mathcal{Z}$-point.  For
$O_i=\operatorname{Int}_AB_i$, Proposition \ref{Prop: relative cap} gives
\[
                         O_i-F\xrightarrow{\simeq}O_i.
\]
The set $O_i$ is saturated for $q$.  In the commutative diagram
\[
\begin{CD}
 O_i-F @>>> O_i\\
 @V{\cong}VV @VV{q}V\\
 q(O_i)-\{p\} @>>> q(O_i),
\end{CD}
\]
the left vertical map is a homeomorphism, and, by Lacher's
Theorem~1.2(c), the right vertical map is a proper homotopy equivalence
\cite[pp.~719--720]{Lac69}.  It follows that the bottom map is a homotopy
equivalence.

We now pass from this neighborhood basis to arbitrary neighborhoods.  Since
$Y$ is an ANR, it is locally contractible.  Given an open neighborhood
$U$ of $p$ and $k\geq0$, choose $i$ and a neighborhood
$p\in V\subset q(O_i)\subset U$ such that $V\hookrightarrow q(O_i)$ is
null-homotopic.  Every map $S^k\to V-\{p\}$ is null-homotopic in
$q(O_i)$.  Since $q(O_i)-\{p\}\to q(O_i)$ is a homotopy equivalence, the
map is null-homotopic in $q(O_i)-\{p\}\subset U-\{p\}$.  Thus $\{p\}$ is
$k$-LCC for every $k$.  Henderson's local characterization
\cite[Thm.~I.1, pp.~206--208]{Hen75} shows that $p$ is a
$\mathcal{Z}$-point of $Y$.

After deleting $p$, the map $q$ is a homeomorphism and
\(Z-\{p\}\cong Q-C\).
The latter is a collared part of the boundary of the manifold $A-F$.
Pushing this boundary collar inward shows that deleting $Z-\{p\}$ from
any inverse-open set in $Y-\{p\}$ is a homotopy equivalence.  Successively
delete the collared set and the $\mathcal{Z}$-point $p$.  The same
inverse-open characterization gives
\(Z\in\mathcal{Z}(Y)\).
This proves the first assertion and \eqref{Eq: W definition}.

Since $Z$ is a $\mathcal{Z}$-set, $W\hookrightarrow Y$ is a homotopy
equivalence.  The quotient map $q:A\to Y$ is also a homotopy equivalence,
and the triangle formed by these two maps and the natural inclusion
$W\hookrightarrow A$ commutes.  The two-out-of-three property proves
\eqref{Eq: W to A}.  All maps preserve the fixed reference homomorphism to
$C_2$, so the equivalence is marked.
\end{proof}

\subsection{The $\mathcal{Z}$-boundary}

The compactification $Y$ is an ANR, but its $\mathcal{Z}$-boundary need
not be.  We now verify the asserted properties of $Z=Q/C$.

\begin{proposition}\label{Prop: non-ANR homology boundary}
The space $Z=Q/C$ is an integral homology $3$-manifold which is not an
ANR.
\end{proposition}

\begin{proof}
The normal subgroups $N_r$ in \eqref{Eq: normal closure tower} satisfy
\[
 N_{r+1}\subseteq N_r,\qquad
 N_r\subseteq F_m^{(r)},\qquad
 \rho(N_r)=K\ne1.
\]
The last paragraph of the proof of Proposition
\ref{Prop: relative cap} shows, using residual solvability of $F_m$, that
$C$ is not nearly $1$-movable.

Embed the standard handlebody $H_0$ in $\mathbb R^3$.  The same nested
handlebodies $H_i$ and the normal-closure calculation
\eqref{Eq: normal tower derived} show that $C$ is not nearly
$1$-movable in this Euclidean embedding.  By the contrapositive of Shrikhande's Theorem~2.2
\cite[pp.~121--123]{Shr83}, the implication
\[
 \mathbb R^3/C \text{ locally simply connected at the collapsed point}
 \quad\Longrightarrow\quad
 C\text{ nearly }1\text{-movable}
\]
shows that $\mathbb R^3/C$ is not locally simply connected at the
point obtained by collapsing $C$.

Choose a neighborhood of $C$ whose closure lies in
$\operatorname{Int}H_0$.  Since the same pair $(H_0,C)$ occurs in both
the chosen embedding $H_0\subset\mathbb R^3$ and the original
inclusion $H_0\subset Q$, collapsing $C$ gives homeomorphic
neighborhoods of the collapsed point in $\mathbb R^3/C$ and in $Q/C$.
Under the identification \(Z\cong Q/C\),
the point obtained by collapsing $C$ corresponds to $p$.  Therefore
$Z$ is not locally simply connected at $p$.  Since metrizable ANRs
are locally contractible, $Z$ is not an ANR.

It remains to calculate the local homology.  Strong excision and
Alexander--Lefschetz duality \cite[Cor.~16.19]{Bre97} give
\[
\begin{aligned}
 H_i(Z,Z-\{p\};\mathbb Z)
 &\cong H_i(Q,Q-C;\mathbb Z)\\
 &\cong \check H^{3-i}(C;\mathbb Z).
\end{aligned}
\]
Since $C$ is connected and \v{C}ech-acyclic, the last group is
$\mathbb Z$ when $i=3$ and is zero otherwise.  Every point of
$Z-\{p\}\cong Q-C$ has an unchanged $3$-manifold neighborhood.
Therefore $Z$ is an integral homology $3$-manifold.
\end{proof}

\begin{proposition}\label{Prop: boundary shape}
The $\mathcal{Z}$-boundary and its connected reference double cover satisfy
\begin{equation}\label{Eq: two shape identities}
             \Sh(Q/C)=\Sh(\mathbb{R}P^3),
             \qquad
             \Sh(\widetilde Z)=\Sh(S^3).
\end{equation}
\end{proposition}

\begin{proof}
Put
\[
 K_i=Q/H_i,
 \qquad
 b_i:K_{i+1}\longrightarrow K_i.
\]
The quotient maps $Q\to Q/H_i$ give the standard inverse-system
expansion of $Q/C$ used in shape theory; thus $\{K_i,b_i\}$ represents the
shape of $Q/C$.  The preimage
of $H_i$ in $\Sigma$ is a disjoint union $H_i^+\sqcup H_i^-$, since the
image of $\pi_1H_i$ is contained in the covering subgroup $K$.  Define
\[
 \widetilde K_i=\Sigma/(H_i^+,H_i^-),
\]
where the two handlebodies are collapsed to two distinct points.  These spaces, with the lifted bonding maps, similarly represent the
shape of $\widetilde Z$.

Kwasik and Schultz construct an equivariant degree-one map
\(f:\Sigma\to S^3.
\)
See \cite[proof of Thm.~2.1, p.~448]{KS88}.
It descends to a degree-one map
\(\overline f:Q\to\mathbb{R}P^3
\)
which induces $\kappa$ on fundamental groups.  The restriction
$\overline f|H_0$ is null-homotopic: the handlebody $H_0$ has the homotopy
type of a finite graph, and $\overline f$ kills the image of its
fundamental group.  Since $(Q,H_0)$ is a cofibration pair, the homotopy
extension property permits us to assume that $\overline f$ is constant on
$H_0$.  It therefore factors into compatible maps
\[
 \overline f_i:K_i\longrightarrow\mathbb{R}P^3,
 \qquad
 \overline f_i b_i=\overline f_{i+1}.
\]
Their lifts are compatible degree-one maps
\(f_i:\widetilde K_i\to S^3\).

Let $m$ be the genus of $H_i$.  Collapsing the two components
$H_i^+\sqcup H_i^-$ separately gives, in dimensions at least two, the
relative sequence of the pair
$(\Sigma,H_i^+\sqcup H_i^-)$.  Since $\Sigma$ is a homology $3$-sphere and
each $H_i^\pm$ is a genus-$m$ handlebody, this sequence gives
\[
 \pi_1\widetilde K_i=1,\qquad
 H_3(\widetilde K_i;\mathbb Z)=\mathbb Z,\qquad
 H_2(\widetilde K_i;\mathbb Z)=\mathbb Z^{2m},\qquad
 H_1(\widetilde K_i;\mathbb Z)=0.
\]
Indeed, the relative boundary homomorphism identifies
\[
 H_2(\widetilde K_i;\mathbb Z)
   \cong H_1(H_i^+;\mathbb Z)\oplus H_1(H_i^-;\mathbb Z),
\]
while the fundamental class of $\Sigma$ gives $H_3(\widetilde K_i)=\mathbb Z$.
For the fundamental-group assertion, observe that the image of either
lifted handlebody normally generates $\pi_1\Sigma=K$; van Kampen's theorem
therefore gives $\pi_1\widetilde K_i=1$ and hence
$H_1(\widetilde K_i)=0$.  Naturality identifies the homomorphism induced
by the lifted bond on $H_2$ with the direct sum of the maps
\[
       H_1(H_{i+1}^\pm)\longrightarrow H_1(H_i^\pm).
\]
Both maps are $\phi_{\mathrm{ab}}=0$.  Therefore the lifted bond is zero on
$H_2$ and has degree one on $H_3$.

Each $\widetilde K_i$ is a simply connected finite $3$-complex.
Whitehead's certain exact sequence makes
$\pi_3\widetilde K_i\to H_3\widetilde K_i$ onto.  Choose a map
$t_i:S^3\to\widetilde K_i$ representing the fundamental class and maps of
$2$-spheres representing a basis of $H_2$.  Their wedge is a homology
equivalence between simply connected complexes, and hence
\begin{equation}\label{Eq: lifted stage wedge}
              \widetilde K_i\simeq
                  S^3\vee\bigvee^{2m}S^2.
\end{equation}
Orient $t_i$ so that $f_it_i$ has degree one and is homotopic to the
identity of $S^3$.  Define
\[
                 s_i=\widetilde b_i t_{i+1}:
                       S^3\longrightarrow\widetilde K_i.
\]
The maps $\widetilde b_i$ and $s_if_{i+1}$ agree on the $S^3$ summand in
\eqref{Eq: lifted stage wedge}.  They are both null-homotopic on every
$S^2$ summand: for the first map this follows from the vanishing on $H_2$
and the Hurewicz theorem, while the second assertion follows from
$\pi_2S^3=0$.  Thus
\[
      \widetilde b_i\simeq s_if_{i+1},
      \qquad f_is_i\simeq\operatorname{Id}_{S^3},
      \qquad \widetilde b_i s_{i+1}\simeq s_i.
\]
For the last identity, $t_{i+1}$ and
$\widetilde b_{i+1}t_{i+2}$ represent the same class in
$H_3(\widetilde K_{i+1})$.  Their difference in $\pi_3$ lies in the
Whitehead term generated by $\pi_2\widetilde K_{i+1}$.  Naturality of
Whitehead's sequence and the fact that $\widetilde b_i$ kills $\pi_2$
show that $\widetilde b_i$ kills this difference.  Hence
$\widetilde b_is_{i+1}\simeq s_i$.  Together these identities show, after passing one stage when necessary,
that the maps $f_i$ and $s_i$ define mutually inverse pro-homotopy
equivalences between $\{\widetilde K_i\}$ and the constant $S^3$ system.

Finally, every $\overline f_i$ induces an isomorphism on fundamental
groups.  In dimensions at least two, its map on pro-homotopy groups is
identified with the lifted pro-equivalence.  The pointed pro-complex
Whitehead theorem \cite[Thm.~3.1]{EG75} makes
$\{\overline f_i\}$ a pro-homotopy equivalence to the constant
$\mathbb{R}P^3$ system.  This proves \eqref{Eq: two shape identities}.
\end{proof}

\section{Identification with the Kwasik--Schultz end}

\subsection{The weak collar of the auxiliary end}

Let $R=\mathbb Z[C_2]$.  Homology with coefficients in $R$ will always
refer to the fixed homomorphism to $C_2$; equivalently, it is ordinary
integral homology in the reference double cover.

Choose numbers $\epsilon_i\searrow0$.  After rounding corners, put
\[
 \mathcal P_i=B_i\cup\operatorname{collar}_{\epsilon_i}(Q),
 \qquad
 E_i=\mathcal P_i-(F\cup Q).
\]
The collar is chosen so that the sets $E_i$ are nested and cofinal among
the neighborhoods of the deleted set $F\cup Q$.  Indeed, strict nesting of
the $B_i$ and the shrinking collar widths imply that $A-\mathcal P_i$ has
compact closure in $W$ and that the $E_i$ are cofinal.  The set $O_i-F$ is connected by
Proposition \ref{Prop: relative cap} and meets the connected collar part
$Q\times(0,\epsilon_i]$.  Thus every $E_i$ is connected, and $W$ has one
end.

\begin{proposition}\label{Prop: cofinal end system}
For every $i$, the inclusion $E_i\hookrightarrow\mathcal P_i$ is a marked
homotopy equivalence and
\[
                         \pi_1\mathcal P_i\cong C_2.
\]
Under these identifications, the $C_2$-marked pro-fundamental group of the end
of $W$ is
\begin{equation}\label{Eq: stationary end pi1}
 C_2\xleftarrow{\operatorname{Id}}C_2
     \xleftarrow{\operatorname{Id}}C_2
     \xleftarrow{\operatorname{Id}}\cdots.
\end{equation}
The end is tame in the sense of Freedman and Quinn and its Wall finiteness
obstruction at infinity vanishes.
\end{proposition}

\begin{proof}
The space $\mathcal P_i$ is obtained, up to rounding, by attaching the simply
connected ball $B_i$ to a collar of $Q$ along $H_i$.  The Seifert--van
Kampen theorem and \eqref{Eq: external image K} give
\[
 \pi_1\mathcal P_i
 \cong
 \pi_1Q/\left\langle\!\left\langle
       \operatorname{im}\pi_1H_i
       \right\rangle\!\right\rangle
 \cong \pi_1Q/K\cong C_2.
\]

The fundamental-group calculation from Proposition
\ref{Prop: relative cap} also applies to $\mathcal P_i$.  More explicitly, exhaust $\mathcal P_i-F$ by the complements
of the deeper balls $B_j$, where $j>i$.  The kernel of the map to $C_2$ at stage $j$ is normally generated
by the image of $\pi_1J_j$, and the next shell kills all of them.  Therefore
$\pi_1(\mathcal P_i-F)\to\pi_1\mathcal P_i$ is an isomorphism.  The reference cover splits
over $F$, and the same relative Alexander--Lefschetz duality calculation
as in \eqref{Eq: relative duality deletion} proves that this map is an
$R$-homology equivalence.  Whitehead's theorem now shows that
$\mathcal P_i-F\hookrightarrow\mathcal P_i$ is a marked homotopy equivalence.  In
$\mathcal P_i-F$, the set $Q-C$ is collared.  Pushing this collar inward gives
\(E_i\xrightarrow{\simeq}\mathcal P_i\).
The inclusion $E_{i+1}\hookrightarrow E_i$ preserves the fixed
$C_2$-marking.  Since the induced maps
$\pi_1E_i\to C_2$ are isomorphisms, every bonding homomorphism is the
identity after these identifications.  This proves \eqref{Eq: stationary end pi1}.

Each $\mathcal P_i$ is a compact ANR of finite CW type, and hence the same is true
up to homotopy for $E_i$.  Thus the end is inward tame.  Together with the
stationary finitely presented pro-fundamental group
\eqref{Eq: stationary end pi1}, Quinn's equivalence of tameness
conventions in \cite[Section~11.9A]{FQ90} makes the end tame in the sense
used in Section~11.9B of that book.  The finite CW models have zero Wall
finiteness obstruction, so the stable end obstruction is zero.
\end{proof}

The group $C_2$ is good in the sense of Freedman and Quinn
\cite[pp.~103, 210]{FQ90}.  The weak end theorem
\cite[Section~11.9B]{FQ90} therefore gives a closed cofinal
weak-collar neighborhood
\begin{equation}\label{Eq: weak collar V}
 V\subset W,
 \qquad N=\partial V,
 \qquad H_*(V,N;R)=0.
\end{equation}
Choose $V$ beyond the stable stage and disjoint from the compact boundary
$P_0$.  The last equality in \eqref{Eq: weak collar V} is the weak-collar
observation immediately preceding that theorem, while
\cite[Section~11.9C(1)]{FQ90} shows that its boundary data
consist of an epimorphism
\(\pi_1N\twoheadrightarrow C_2
\)
with perfect kernel.

\subsection{The complementary cobordism and the actual Kwasik--Schultz end}

We next compare the auxiliary end with the actual Kwasik--Schultz end.  The
first step allows the cap to remain in the signature-bearing compact piece $A$ rather
than in a periodic block.

\begin{proposition}
\label{Prop: complementary cobordism}
After a bicollar adjustment, put
\[
        K_c=\overline{W-V}^{\,W},
        \qquad L=\overline{A-K_c}^{\,A}.
\]
Then $L$ is an ordinary compact topological $4$-manifold cobordism with
\begin{equation}\label{Eq: L boundary}
                         \partial L=N\sqcup Q.
\end{equation}
Both boundary pairs are $R$-acyclic:
\[
                         H_*(L,N;R)=H_*(L,Q;R)=0.
\]
Thus $L$ is an $R$-homology $h$-cobordism of the weak-collar data
$(N\to C_2)$ and $(Q\xrightarrow{\kappa}C_2)$ in the sense of
Freedman--Quinn Section~11.9C(3).
\end{proposition}

\begin{proof}
The compact set $K_c$ lies in $W$, is disjoint from $F\cup Q$, and contains
$P_0$.  After moving its frontier through a bicollar, $K_c$ is a compact
codimension-zero submanifold.  Hence
\[
 W=K_c\cup_NV,
 \qquad
 A=K_c\cup_NL.
\]
Although $L$ contains the wild compactum $F$, it is the ambient
complementary closure of an ordinary codimension-zero submanifold of $A$.
It is therefore an ordinary compact $4$-manifold, and
\eqref{Eq: L boundary} follows.

The marked homotopy equivalence $W\hookrightarrow A$ in
\eqref{Eq: W to A} is the identity on $K_c$.  The long exact sequences of
pairs give
\[
             H_*(W,K_c;R)\xrightarrow{\cong}H_*(A,K_c;R).
\]
Bicollared excision and \eqref{Eq: weak collar V} now give
\begin{equation}\label{Eq: L N calculation}
\begin{aligned}
 H_*(A,K_c;R)
  &\cong H_*(W,K_c;R)\\
  &\cong H_*(V,N;R)=0,\\
 H_*(L,N;R)&\cong H_*(A,K_c;R)=0.
\end{aligned}
\end{equation}

The $C_2$-marking $\pi_1L\to C_2$ is onto because its restriction
to $Q$ is $\kappa$.  Let
$(\widehat L;\widehat N,\widehat Q)$ be the connected reference double
cover.  It is oriented because it is the restriction of the oriented
reference cover $\widetilde A$.  Equation \eqref{Eq: L N calculation} says
\[
                  H_*(\widehat L,\widehat N;\mathbb Z)=0.
\]
The integral universal-coefficient theorem and complementary-boundary
Poincar\'e--Lefschetz duality in this oriented cover give
\[
 H_j(\widehat L,\widehat Q;\mathbb Z)
 \cong H^{4-j}(\widehat L,\widehat N;\mathbb Z)=0.
\]
Returning to the $R$-homology notation gives $H_*(L,Q;R)=0$ and
completes the proof.
\end{proof}

Since $\Sigma$ is an integral homology sphere, $K=\pi_1\Sigma$ is perfect.
Thus $\kappa$ is an epimorphism with perfect kernel, exactly the hypothesis of \cite[Section~11.9C(1)]{FQ90}.  Let $V_Q$ be the standard
weak collar which realizes the data $(Q,\kappa)$.

\begin{proposition}\label{Prop: end germ identification}
Let $E_{KS}$ denote either one-sided Kwasik--Schultz end beginning at $Q$.
Then
\begin{equation}\label{Eq: two germ identifications}
                  \germ W\cong\germ V_Q\cong\germ E_{KS}.
\end{equation}
\end{proposition}

\begin{proof}
Apply \cite[Section~11.9C(3)]{FQ90} to the weak collars $V$ and $V_Q$
and to the $R$-homology $h$-cobordism $L$.  That theorem says
that $L$ is $R$-homology $h$-cobordant rel boundary to a cobordism induced
by a homeomorphism of neighborhoods of ends.  Consequently
\begin{equation}\label{Eq: W standard germ}
                         \germ W\cong\germ V_Q.
\end{equation}
Notice that this conclusion transports only the end germ.  It does not
carry the cocore disks in $A$ into a periodic Kwasik--Schultz block.

We next consider the actual Kwasik--Schultz end.  In the proof of Kwasik--Schultz
Theorem~2.1, the periodic block $T(V,\alpha)$ is the domain of an
equivariant homotopy equivalence to $S^3\times I$, and its boundary
restrictions are the equivariant degree-one maps from $\Sigma$ to $S^3$
\cite[p.~448]{KS88}.  Hence both boundary inclusions upstairs induce
integral homology isomorphisms, and each quotient block is an
$R$-homology cobordism.  The periodic decomposition is also described in
\cite[Cor.~3.4, pp.~450--451]{KS88}.  Singular homology commutes with the
directed union of the finite stacks, and therefore
\begin{equation}\label{Eq: actual KS relative homology}
                         H_*(E_{KS},Q;R)=0.
\end{equation}
The infinite-stack construction on \cite[p.~448]{KS88} gives the proper
finite-type behavior required for tameness, and Kwasik and Schultz state
on \cite[p.~449]{KS88} that the end groups are stable and finite.  In the
present case the stable $C_2$-marked group is $C_2$.  The Wall end obstruction
is zero by the vanishing projective obstruction in
\cite[Thm.~2.1(2), p.~447]{KS88}.  Therefore the weak end theorem
\cite[Section~11.9B]{FQ90} supplies a weak-collar tail
$V'\subset E_{KS}$ with bicollared frontier.  Put
\[
 N'=\partial V',
 \qquad
 J=\overline{E_{KS}-V'}^{\,E_{KS}}.
\]
Excision gives
\[
       H_*(E_{KS},J;R)\cong H_*(V',N';R)=0.
\]
The triple $Q\subset J\subset E_{KS}$, together with
\eqref{Eq: actual KS relative homology}, gives
\[
                         H_*(J,Q;R)=0.
\]
Complementary-boundary duality in the connected reference double cover
then gives
\[
                         H_*(J,N';R)=0.
\]
The inherited map $\pi_1J\to C_2$ is onto because its restriction to $Q$
is $\kappa$, and its restrictions to $Q$ and $N'$ are the two weak-collar
data maps.  Thus $J$ satisfies all the hypotheses of \cite[Section~11.9C(3)]{FQ90}, and the germ of $V'$ is
homeomorphic to the germ of $V_Q$.  Together with
\eqref{Eq: W standard germ}, this proves
\eqref{Eq: two germ identifications}.
\end{proof}

We use the following relative form of the $\mathcal{Z}$-set push in the
final gluing.

\begin{lemma}\label{Lemma: relative Z push}
Let $Z$ be a $\mathcal{Z}$-set in a compact metric ANR $Y$, and let $D$ be
a closed set disjoint from $Z$.  Given a neighborhood $U$ of $Z$ whose
closure is disjoint from $D$, there are a neighborhood $U_1$ of $Z$ with
$\overline U_1\subset U$ and a homotopy which starts at the identity,
instantly pushes $Y$ off $Z$, is fixed on $Y-U_1$, and carries
$U_1\times[0,1]$ into $U$.  In particular, it is fixed on a neighborhood
of $D$.
\end{lemma}

\begin{proof}
Choose open sets $U_0$ and $U_1$ such that
\[
 Z\subset U_0\subset\overline U_0\subset U_1
 \subset\overline U_1\subset U.
\]
The small-homotopy characterization of $\mathcal{Z}$-sets
\cite[Thm.~I.1, pp.~206--208]{Hen75}, applied to an open refinement of
$\{U,Y-\overline U_1\}$, gives a homotopy
\begin{equation}\label{Eq: instant Z homotopy}
 H:Y\times[0,1]\longrightarrow Y,
 \qquad H_0=\operatorname{Id}_Y,
 \qquad H_t(Y)\subseteq Y-Z\quad(t>0),
\end{equation}
whose tracks beginning in $U_1$ remain in $U$.  Since $Y$ is metric,
Urysohn's lemma supplies a continuous function $\lambda:Y\to[0,1]$ which
is one on $\overline U_0$ and zero on $Y-U_1$.  Define
\begin{equation}\label{Eq: cutoff Z homotopy}
                         \overline H(x,t)=H(x,t\lambda(x)).
\end{equation}
If $\lambda(x)=0$, then $\overline H(x,t)=x\notin Z$ unless $x\in Z$;
but $\lambda=1$ on $Z$.  If $\lambda(x)>0$ and $t>0$, then
\eqref{Eq: instant Z homotopy} also shows that $\overline H(x,t)\notin Z$.
Thus \eqref{Eq: cutoff Z homotopy} instantly misses $Z$ and is fixed on
$Y-U_1$.  Its tracks beginning in $U_1$ lie in $U$, and
$Y-\overline U$ is a fixed open neighborhood of $D$.
\end{proof}

\begin{lemma}\label{Lemma: compact core grafting}
Suppose $V_X$ and $V_W$ are homeomorphic bicollared cofinal tails of
manifold ends, where $V_W\subset W$.  Then the
$\mathcal{Z}$-compactification in Proposition \ref{Prop: compact quotient}
may be transported to $V_X$ and attached to the untouched compact core.
The resulting compactification is a compact ANR of covering dimension at
most four and has $\mathcal{Z}$-boundary $Z$.
\end{lemma}

\begin{proof}
Let $K_W$ be the compact complement of $V_W$ in $W$.  The closure of the
chosen tail in $Y$ is
\[
                  \overline V_W^{\,Y}
                       =Y-\operatorname{Int}K_W.
\]
A bicollar of $\partial K_W$ makes this closure a neighborhood retract of
an open subset of the ANR $Y$.  It is therefore a compact ANR.  Choose a
neighborhood $U$ of $Z$ whose closure lies in the interior of
$\overline V_W^{\,Y}$ and misses the frontier collar.  Lemma
\ref{Lemma: relative Z push} gives a controlled $\mathcal{Z}$-push: points
in its support have their tracks in $U$, and all other points are fixed.
It therefore restricts to a self-homotopy of
$\overline V_W^{\,Y}$ which instantly misses $Z$ and is fixed on the
frontier collar.

Transport $\overline V_W^{\,Y}$ to the corresponding compactified tail of
$V_X$, and attach the compact core of the given manifold along the
bicollared frontier.  Both pieces are compact ANRs, their common frontier
is a compact $3$-manifold, and the bicollars make its inclusions
cofibrations.  Hanner's adjunction theorem \cite[Thm.~8.2]{Han51} makes
the union a compact ANR.  The relative
$\mathcal{Z}$-push extends by the identity over the compact core, so $Z$
remains a $\mathcal{Z}$-set.  Finally, the finite closed-sum theorem
\cite{HW41} gives covering dimension at most four.
\end{proof}

\begin{proof}[Proof of Theorem \ref{Th: KS Z-compactification}]
Proposition \ref{Prop: compact quotient} gives a compact ANR
$\mathcal{Z}$-compactification of the auxiliary end $W$.  By Proposition
\ref{Prop: end germ identification}, choose bicollared cofinal tails in
$W$ and in either one-sided Kwasik--Schultz end on which the germ
homeomorphism is represented by an actual homeomorphism.  Apply Lemma
\ref{Lemma: compact core grafting}.  This produces the required
$\mathcal{Z}$-compactification at one end.

Perform the construction independently in two disjoint cofinal
neighborhoods of the two ends of $M_\alpha$, and attach the untouched
compact middle.  The resulting remainder is
\(Z_-\sqcup Z_+\).
The relative pushes on the two end compactifications extend by the identity
over the compact middle, so their disjoint union is a $\mathcal{Z}$-set.
The ANR and dimension conclusions follow from the same adjunction and
finite closed-sum arguments.  Propositions
\ref{Prop: non-ANR homology boundary} and \ref{Prop: boundary shape} give the
remaining assertions about each $\mathcal{Z}$-boundary component.  This completes the
proof.
\end{proof}

\begin{remark}[The signature class]\label{Rem: signature class}
The compact cobordism $A$ is not an $R$-homology product.  In its reference cover,
\[
 H_2(A,Q;R)
 \cong H_2(\widetilde A,\Sigma;\mathbb Z)
 \cong H_2(P^4(V,\alpha);\mathbb Z)\ne0,
\]
and this group carries the transferred signature $8$.  This
signature-bearing group lies in the compact core.  Indeed, relative
Mayer--Vietoris for
$A=K_c\cup_NL$, together with $H_*(L,Q;R)=0$, gives
\[
                         H_*(A,Q;R)\cong H_*(K_c,N;R).
\]
Thus the signature-bearing homology remains in the compact core $K_c$.
Only the outer cobordism $L$ is an $R$-homology $h$-cobordism of its
boundary data.
\end{remark}

We recall the notion of a pseudo-collar.  A manifold neighborhood
of infinity $N$ in a manifold $M$ is a \emph{homotopy collar} if the
inclusion
\(
        \operatorname{Fr}_M N\hookrightarrow N
\)
is a homotopy equivalence.  A \emph{pseudo-collar} is a homotopy collar
which contains arbitrarily small homotopy collar neighborhoods of
infinity, and $M$ is \emph{pseudo-collarable} if it contains a
pseudo-collar neighborhood of infinity.  The notion was developed by
Guilbault and Guilbault--Tinsley in their study of manifolds with
nonstable fundamental group at infinity; see
\cite{Gui00,GT03,GT06}.  A complete characterization of
pseudo-collarability for high-dimensional manifolds, including
manifolds with noncompact boundary, was obtained in \cite{Gu20}.

\begin{proof}[Proof of Corollary \ref{Cor: KS non-pseudo-collar}]
Kwasik and Schultz show that neither end of $M_\alpha$ admits a
$1$-neighborhood; see \cite[p.~449]{KS88}.  On the other hand, the
fundamental group at either end is stable and equal to $C_2$, which is
good in the sense of Freedman--Quinn.  By
\cite[Proposition~2.3]{GT06}, a $4$-dimensional pseudo-collar with good
stable fundamental group contains an open collar neighborhood of
infinity.  Such a neighborhood is, in particular, a $1$-neighborhood.
Thus neither end of $M_\alpha$ can be pseudo-collarable.  Theorem
\ref{Th: KS Z-compactification} shows that $M_\alpha$ is
$\mathcal Z$-compactifiable.
\end{proof}

\section*{Acknowledgements}
The author used OpenAI Codex with the GPT-5.6 Sol Ultra model during
the development and preparation of this manuscript.  The project
involved approximately 73 hours of Codex runs.  The author first
provided the model with a sketch and the relevant background, including the work of
Kwasik--Schultz and Freedman--Quinn, and used it as an interactive
research tool to explore possible constructions of a
$\mathcal{Z}$-compactification.  Codex assisted in developing and
refining the relative-cap construction, checking intermediate
arguments, locating and comparing relevant results in the literature,
and improving the organization and exposition of the proof.  The
resulting arguments were subsequently checked, substantially revised,
and rewritten by the author.  All mathematical statements, proofs, and
references in the final manuscript were independently verified by the
author, who takes full responsibility for the content.

The author was supported by NSFC Grant
No.~12201102.

\end{document}